\documentclass[10pt,oneside]{amsart}

\usepackage{amssymb}
\usepackage{amscd}

\title{Complete determination of the number of Galois points for a smooth plane curve}
\author{Satoru Fukasawa}

\address{Department of Mathematical Sciences,  
Faculty of Science, Yamagata University, 
Kojirakawa-machi 1-4-12, Yamagata 990-8560, Japan.}
\email{s.fukasawa@sci.kj.yamagata-u.ac.jp} 

\newtheorem{theorem}{Theorem}
\newtheorem{proposition}{Proposition}

\newtheorem{fact}{Fact}

\newtheorem{remark}{Remark}
\newtheorem{lemma}{Lemma} 
\newtheorem{problem}{}

\begin{document}
\begin{abstract} 
Let $C$ be a smooth plane curve. 
A point $P$ in the projective plane is said to be Galois with respect to $C$ if the function field extension induced from the point projection from $P$ is Galois. 
We denote by $\delta(C)$ (resp. $\delta'(C)$) the number of Galois points contained in $C$ (resp. in $\mathbb P^2 \setminus C$). 
In this article, we determine the numbers $\delta(C)$ and $\delta'(C)$ in any remaining open cases.   
Summarizing results obtained by now, we will have a complete classification theorem of smooth plane curves by the number $\delta(C)$ or $\delta'(C)$.   
In particular, we give new characterizations of Fermat curve and Klein quartic curve by the number $\delta'(C)$. 
\end{abstract}
\maketitle
\section{Introduction}  
Let the base field $K$ be an algebraically closed field of characteristic $p \ge 0$ and let $C \subset \mathbb P^2$ be a smooth plane curve of degree $d \ge 4$. 
In 1996, H. Yoshihara introduced the notion of {\it Galois point} (see \cite{miura-yoshihara, yoshihara} or survey paper \cite{fukasawa3}). 
If the function field extension $K(C)/K(\mathbb P^1)$, induced from the projection $\pi_P:C \rightarrow \mathbb P^1$ from a point $P \in \mathbb P^2$, is Galois, then the point $P$ is said to be Galois with respect to $C$. 
When a Galois point $P$ is contained in $C$ (resp. $\mathbb P^2 \setminus C$), we call $P$ an inner (resp. outer) Galois point. 
We denote by $\delta(C)$ (resp. $\delta'(C)$) the number of inner (resp. outer) Galois points for $C$. 
It is remarkable that many classification results of algebraic varieties have been given in the theory of Galois point.

Yoshihara determined $\delta(C)$ and $\delta'(C)$ in characteristic $p=0$ (\cite{miura-yoshihara, yoshihara}). 
In characteristic $p>0$, M. Homma \cite{homma3} settled $\delta(H)$ and $\delta'(H)$ for a Fermat curve $H$ of degree $p^e+1$.  
Recently, the present author determined $\delta(C)$ when $p>2$ or $d-1$ is not a power of $2$ (\cite{fukasawa1, fukasawa2}), and $\delta'(C)$ when $d$ is not divisible by $p$, $d=p$, or $d=2^e$ in $p=2$ (\cite{fukasawa1, fukasawa2, fukasawa5}). 
The following problems remain open (\cite[Part III, Problem]{fukasawa2}, \cite[Problem 2]{fukasawa3}). 

\begin{problem} 
\begin{itemize}
\item[(1)] Let $p=2$ and let $e \ge 2$. 
Find and classify smooth plane curves of degree $d=2^e+1$ with $\delta(C)=d$. 
\item[(2)] Let $p > 0$, $e \ge 1$ and let $d=p^el$, where $l$ is not divisible by $p$. 
Assume that $(p^e, l) \ne (p, 1), (2^e, 1)$. 
Then, determine $\delta'(C)$.
\end{itemize}  
\end{problem}

In this article, we give a complete answer to these problems. 
 
\begin{theorem} \label{innerGP}
Let $p=2$, let $e \ge 2$ and let $C$ be a smooth plane curve of degree $d=2^e+1$. 
Then, $\delta(C)=d$ if and only if $C$ is projectively equivalent to a plane curve given by
\begin{equation} \label{d-Galois}
 \prod_{\alpha \in \mathbb F_{2^e}}(x+\alpha y+\alpha^2)+cy^{2^e+1}=0, \tag{1c} 
\end{equation} 
where $c \in K \setminus \{0, 1\}$. 
\end{theorem} 

\begin{theorem} \label{outerGP}
Let the characteristic $p >0$, let $e \ge 1$, let $l$ be not divisible by $p$, and let $C$ be a smooth plane curve of degree $d=p^el \ge 4$. 
If $(p^e, l)\ne (2^e, 1)$, then $\delta'(C)\le 1$. 
\end{theorem}

Summarizing Theorems \ref{innerGP} and \ref{outerGP} and the results of Yoshihara, Homma and the present author, we will have the following classification theorem of smooth plane curves by the number $\delta(C)$ or $\delta'(C)$. 

\begin{theorem}[Yoshihara, Homma, Fukasawa] \label{summary}
Let $C$ be a smooth plane curve of degree $d \ge 4$ in characteristic $p \ge 0$. 
Then: 
\begin{itemize}
\item[(I)] $\delta(C)=0, 1, d$ or $(d-1)^3+1$. 
Furthermore, we have the following. 
\begin{itemize}
\item[(i)] 
$\delta(C)=(d-1)^3+1$ if and only if $p>0$, $d=p^e+1$ and $C$ is projectively equivalent to a Fermat curve. 
\item[(ii)] 
$\delta(C)=d \ge 5$ if and only if $p=2$, $d=2^e+1$ and $C$ is projectively equivalent to a curve defined by $\prod_{\alpha \in \mathbb F_{2^e}}(x+\alpha y+\alpha^2)+cy^{2^e+1}=0$, where $c \in K \setminus \{0, 1\}$. 
\item[(iii)]$\delta(C)=d=4$ if and only if $p \ne 2, 3$ and $C$ is projectively equivalent to a curve defined by $x^3+y^4+1=0$. 
\end{itemize}
\item[(II)] $\delta'(C)=0, 1, 3, 7$ or $(d-1)^4-(d-1)^3+(d-1)^2$. 
Furthermore, we have the following. 
\begin{itemize}
\item[(i)] $\delta'(C)=(d-1)^4-(d-1)^3+(d-1)^2$ if and only if $p>0$, $d-1$ is a power of $p$ and $C$ is projectively equivalent to a Fermat curve. 
\item[(ii)] 
$\delta'(C)=7$ if and only if $p=2$, $d=4$ and $C$ is projectively equivalent to Klein quartic curve. 
\item[(iii)] 
$\delta'(C)=3$ and three Galois points are not contained in a common line if and only if $d$ is not divisible by $p$, $d-1$ is  not a power of $p$, and $C$ is projectively equivalent to a Fermat curve. 
\item[(iv)] 
$\delta'(C)=3$ and three Galois points are contained in a common line if and only if $p=2$, $d=4$ and $C$ is projectively equivalent to a plane curve defined by 
$$ (x^2+x)^2+(x^2+x)(y^2+y)+(y^2+y)^2+c=0,$$
where $c \in K\setminus \{0,1\}$. 
\end{itemize} 
\end{itemize}  
\end{theorem} 

This is a modified and extended version of the paper \cite[Part IV]{fukasawa2} (which will have been published only in arXiv).

\section{Preliminaries} 
Let $C \subset \mathbb P^2$ be a smooth plane curve of degree $d \ge 4$ in characteristic $p>0$. 
For a point $P \in C$, we denote by $T_PC \subset \mathbb P^2$ the (projective) tangent line at $P$. 
For a projective line $l \subset \mathbb P^2$ and a point $P \in C \cap l$, we denote by $I_P(C, l)$ the intersection multiplicity of $C$ and $l$ at $P$.  
We denote by $\overline{PR}$ the line passing through points $P$ and $R$ when $P \ne R$, and by $\pi_P: C \rightarrow \mathbb P^1; R \mapsto \overline{PR}$ the point projection from a point $P \in \mathbb P^2$. 
If $R \in C$, we denote by $e_R$ the ramification index of $\pi_P$ at $R$. 
It is not difficult to check the following.  
 
\begin{lemma} \label{index}
Let $P \in \mathbb P^2$ and let $R \in C$. 
Then for $\pi_P$ we have the following.
\begin{itemize}
\item[(1)] If $R=P$, then $e_R =I_R(C, T_RC)-1$. 
\item[(2)] If $R \ne P$, then $e_R=I_R(C, \overline{PR})$. 
\end{itemize} 
\end{lemma}

Let $P$ be a Galois point. 
We denote by $G_P$ the group of birational maps from $C$ to itself corresponding to the Galois group ${\rm Gal}(K(C)/\pi_R^*K(\mathbb P^1))$. 
We find easily that the group $G_P$ is isomorphic to a subgroup of the automorphism group ${\rm Aut}(C)$ of $C$. 
We identify $G_P$ with the subgroup.  

If a Galois covering $\theta:C \rightarrow C'$ between smooth curves is given, then the Galois group $G$ acts on $C$ naturally. 
We denote by $G(R)$ the stabilizer subgroup of $R$. 
The following fact is useful to find Galois points (see \cite[III. 7.1, 7.2 and 8.2]{stichtenoth}). 
\begin{fact} \label{Galois covering} 
Let $\theta: C \rightarrow C'$ be a Galois covering of degree $d$ with a Galois group $G$ and let $R, R' \in C$. 
Then we have the following. 
\begin{itemize}
\item[(1)] For any $\sigma \in G$, we have $\theta(\sigma(R))=\theta(R)$.   
\item[(2)] If $\theta(R)=\theta(R')$, then there exists an element $\sigma \in G$ such that $\sigma(R)=R'$. 
\item[(3)] The order of $G(R)$ is equal to $e_R$ at $R$ for any point $R \in C$. 
\item[(4)] If $\theta(R)=\theta(R')$, then $e_R=e_{R'}$. 
\item[(5)] The index $e_R$ divides the degree $d$. 
\end{itemize}
\end{fact}

We recall a theorem on the structure of the Galois group at a Galois point (see \cite[Part II]{fukasawa2}). 
Let $d-1=p^el$ (resp. $d=p^el$), where $l$ is not divisible by $p$, let $\zeta$ be a primitive $l$-th root of unity, and let $k=[\mathbb F_p(\zeta): \mathbb F_p]$. 
Let $P=(1:0:0)$ be an inner (resp. outer) Galois point for $C$. 
The projection $\pi_P:C \rightarrow \mathbb P^1$ is given by $(x:y:1) \mapsto (y:1)$. 
We have a field extension $K(x,y)/K(y)$ via $\pi_P$. 
Let $\gamma \in G_P$. 
Then, the automorphism $\gamma \in G_P$ can be extended to a linear transformation of $\mathbb P^2$ (see \cite[Appendix A, 17 and 18]{acgh} or \cite{chang}). 
Let $A_{\gamma}=(a_{ij})$ be a $3 \times 3$ matrix representing $\gamma$. 
Since $\gamma \in G_P$, $\sigma^*(y)=y$. 
Then, $(a_{21}x+a_{22}y+a_{23})-(a_{31}x+a_{32}y+a_{33})y=0$ in $K(x, y)$. 
Since $d \ge 4$, we have $a_{21}=a_{23}=a_{31}=a_{32}=0$ and $a_{22}=a_{33}$. 
We may assume that $a_{22}=a_{33}=1$. 
Since $\gamma^{p^el}=1$, we have $a_{11}^l=1$.
We take a group homomorphism $G_P \rightarrow K \setminus 0; \gamma \mapsto a_{11}(\gamma)$, where $a_{11}(\gamma)$ is the $(1, 1)$-element of $A_{\gamma}$.  
Then, we have a splitting exact sequence of groups 
$$ 0 \rightarrow (\mathbb Z/p \mathbb Z)^{\oplus e} \rightarrow G_P \rightarrow \langle \zeta \rangle \rightarrow 1, $$  
and the following theorem. 
\begin{theorem} \label{group}
Let $C \subset \mathbb P^2$ be a smooth curve and let $P$ be an inner (resp. outer) Galois point. 
Then, $k$ divides $e$ and $G_P \cong (\mathbb Z/p\mathbb Z)^{\oplus e} \rtimes \langle \zeta \rangle$. 
 
\end{theorem}

(The condition that $k$ divides $e$ is equivalent to that $l$ divides $p^e-1$.)
We denote the kernel (resp. quotient) by $\mathcal{K}_P$ (resp. by $\mathcal{Q}_P$). 
An element $\sigma \in \mathcal{K}_P$ (resp. a generator $\tau \in \mathcal{Q}_P$) is represented by a matrix 
$$ 
A_{\sigma}=
\left(\begin{array}{ccc}
1 & a_{12}(\sigma) & a_{13}(\sigma) \\
0 & 1 & 0 \\
0 & 0 & 1
\end{array}\right)
(\mbox{ resp. } A_{\tau}= 
\left(
\begin{array}{ccc} 
\zeta & a_{12}(\tau) & a_{13}(\tau)\\
0 & 1 & 0 \\
0 & 0 & 1
\end{array}\right)), 
$$  
where $a_{12}(\sigma), a_{13}(\sigma), a_{12}(\tau), a_{13}(\tau) \in K$.   
For each non-identity element $\gamma \in G_P$, there exist $\sigma \in \mathcal{K}_P$ and $i$ such that $\gamma=\sigma \tau^i$.
Then, there exists a unique line $L_{\gamma}$, which is defined by $(\zeta^i-1)X+(a_{12}(\sigma)+a_{12}(\tau^i))Y+(a_{13}(\sigma)+a_{13}(\tau^i))Z=0$, such that $\gamma(R)=R$ for any $R \in L_{\gamma}$. 
Note that $P \in L_{\gamma}$ if and only if $\gamma \in \mathcal{K}_P$. 
For a suitable system of coordinates, we can take $a_{12}(\tau)=a_{13}(\tau)=0$.

Finally in this section, we note the following facts on automorphisms of $\mathbb P^1$. 

\begin{lemma} \label{automorphisms of P^1} 
We denote by ${\rm Aut}(\mathbb P^1)$ the automorphism group of $\mathbb P^1$. 
\begin{itemize} 
\item[(1)] Let $P_1, P_2, P_3 \in \mathbb P^1$ be three distinct points and let $\gamma_1, \gamma_2 \in {\rm Aut}(\mathbb P^1)$. 
If $\gamma_1(P_i)=\gamma_2(P_i)$ for $i=1, 2, 3$, then $\gamma_1=\gamma_2$.  
\item[(2)] Let $P_1, P_2 \in \mathbb P^1$ be distinct points and let $G \subset {\rm Aut}(\mathbb P^1)$ be a finite subgroup.   
If $\gamma(P_1)=P_1$ and $\gamma(P_2)=P_2$ for any $\gamma \in G$, then $G$ is a cyclic group whose order is not divisible by $p$ if $p>0$.   
\item[(3)] Let $l$ be not divisible by $p$, let $P \in \mathbb P^1$, and let $G \subset {\rm Aut}(\mathbb P^1)$ be a subgroup of order $l$. 
Assume that $G$ is cyclic and $\tau(P)=P$ for any $\tau \in G$. 
Then, there exists a unique point $Q$ such that $Q \ne P$ and $\tau(Q)=Q$ for any $\tau \in G$. 

\end{itemize}
\end{lemma}

\begin{proof} 
The fact (1) is easily proved, if we use the classical fact that any automorphism of $\mathbb P^1$ is a linear transformation. 
We prove (2). 
We may assume that $P_1=(1:0)$ and $P_2=(0:1)$. 
Let $\gamma \in G$. 
Since $\gamma(P_1)=P_1$ and $\gamma(P_2)=P_2$, $\gamma$ is represented by a matrix 
$$ A_{\gamma}=\left(\begin{array}{cc}
a(\gamma) & 0  \\
0 & 1 
\end{array}\right), $$
where $a(\gamma) \in K$. 
Then, a homomorphism $\psi: G \rightarrow K \setminus 0: \gamma \mapsto a(\gamma)$ is injective and $\psi(G)$ is cyclic. 
Let $m$ be the order of $\psi(G)$. 
Then, $\psi(G)$ is contained in a set $\{x \in K \setminus 0 | x^m-1=0\}$. 
If $m$ is divisible by $p$, the set consists of at most $m/p$ elements. 
Therefore, $m$ is not divisible by $p$. 
We have the conclusion. 

We prove (3). 
We may assume that $P=(1:0)$. 
Let $\tau$ be a generator of $G$. 
Since $\tau(P)=P$ and $\tau$ is an automorphism of order $l$ not divisible by $p$, $\tau$ is represented by a matrix 
$$ A_{\tau}=\left(\begin{array}{cc} \zeta & b \\
0 & 1 
\end{array}\right),
$$
where $\zeta$ is a primitive $l$-th root of unity and $b \in K$.
Then, $\tau^i$ is represented by a matrix 
$$ A_{\tau^i}=\left(\begin{array}{cc} \zeta^i & \frac{\zeta^i-1}{\zeta-1}b \\
0 & 1 \end{array}\right). $$
Let $Q=(x:1)$. 
Then, $\tau^i(Q)=Q$ if and only if $(\zeta-1)x+b=0$. 
We have the conclusion. 
\end{proof}

\section{Only-if-part of the proof of Theorem \ref{innerGP}} 
Let $p=2$, let $q=2^e \ge 4$ and let $C$ be a plane curve of degree $d=q+1$. 
Assume that $\delta(C)=d$. 
Let $P_1, \ldots, P_d$ be inner Galois points for $C$. 
By the results of the previous paper \cite[Part III, Lemma 1, Propositions 1, 3 and 4]{fukasawa2}, we have the following. 

\begin{proposition} \label{Basics} 
Assume that $\delta(C)=d$. 
Then, we have the following. 
\begin{itemize}
\item[(1)] Galois points $P_1, \ldots, P_d$ are contained in a common line.
\item[(2)] For any $i$ and any element $\sigma \in G_{P_i} \setminus \{1\}$, the order of $\sigma$ is two. 
\item[(3)] For any $i$ and any elements $\sigma, \tau \in G_{P_i} \setminus \{1\}$ with $\sigma \ne \tau$, $L_{\sigma} \ne T_{P_i}C$ and $L_{\sigma} \ne L_{\tau}$. 
In particular, a set $\{T_{P_1}C \cap T_{P_i}C|2 \le i \le d\}$ consists of exactly $d-1$ points. 
\end{itemize} 
By the condition $(1)$ and Fact $\ref{Galois covering}(2)$, for each $i$ with $3 \le i \le d$, there exists $\tau_i \in G_{P_i}$ such that $\tau_i(P_1)=P_2$. 
Let $\{Q\}=T_{P_1}C \cap T_{P_2}C$. 
In addition, we have the following by the condition $(2)$. 
\begin{itemize}
\item[(4)] For any $i$ with $3 \le i \le d$, $\tau_i(P_2)=P_1$ and $\tau_i(Q)=Q$. 
\item[(5)] For any $i, j$ with $3 \le i, j \le d$, $\tau_i\tau_j(P_1)=P_1$, $\tau_i\tau_j(P_2)=P_2$ and $\tau_i\tau_j(Q)=Q$.
\end{itemize}
\end{proposition}

\begin{lemma} \label{coordinates}
For a suitable system of coordinates, we may assume that $P_1=(1:0:0)$, $P_2=(0:0:1)$ and $Q=(0:1:0)$.  
\end{lemma} 
By Lemma \ref{coordinates} and Proposition \ref{Basics}(2)(4), $\tau_i$ is given by a matrix 
$$ A_{\tau_i}=
\left(\begin{array}{ccc}
0 & 0 & 1 \\
0 & a_i & 0 \\
a_i^2 & 0 & 0
\end{array}\right), 
$$
for some $a_i \in K$. 
Then, $\tau_i\tau_j$ is given by a matrix 
$$ A_{\tau_i\tau_j}=
\left(\begin{array}{ccc}
a_j^2 & 0 & 0 \\
0 & a_ia_j & 0 \\
0 & 0 & a_i^2
\end{array}\right). 
$$
Let $H(C)$ be a subgroup of ${\rm Aut}(\mathbb P^2)$ satisfying that 
\begin{itemize}
\item[(h1)] $\gamma(P_1)=P_1$, $\gamma(P_2)=P_2$ and $\gamma(Q)=Q$, 
\item[(h2)] $\{\gamma(P_i)| 3 \le i \le d\}=\{P_i|3 \le i \le d\}$, and 
\item[(h3)] $\gamma(C)=C$ 
\end{itemize} 
for any $\gamma \in H(C)$. 

\begin{lemma} \label{cyclic}
The group $H(C)$ is a cyclic group whose order is at most $d-2=q-1$. 
\end{lemma}

\begin{proof} 
By the condition (h1) of $H(C)$, for any $\gamma \in H(C)$, $\gamma$ is represented by a matrix
$$ A_{\gamma}=\left(\begin{array}{ccc}
a & 0 & 0 \\
0 & b & 0 \\
0 & 0 & 1
\end{array}\right) $$ 
for some $a, b \in K$. 
We prove that $\gamma$ depends only on the image of $P_3$.
Precisely, we show that for $\gamma_1, \gamma_2 \in H(C)$, if $\gamma_1(P_3)=\gamma_2(P_3)$, then $\gamma_1=\gamma_2$. 
To prove this, it suffices to show that $\gamma=1$ if $\gamma(P_3)=P_3$.  
Assume that $\gamma(P_3)=P_3$. 
Since $\gamma$ fixes three distinct points $P_1, P_2, P_3$ on the line $\overline{P_1P_2}$, $\gamma$ is identity on the line $\overline{P_1P_2}$, by Lemma \ref{automorphisms of P^1}(1) in Section 2.   
We have $a=1$, since $\overline{P_1P_2}$ is given by $Y=0$. 
On the other hand, by the condition (h3), $\gamma(T_{P_3}C)=T_{P_3}C$. 
Then, a point $Q_0$ given by $T_{P_1}C \cap T_{P_3}C$ is fixed by $\gamma$. 
Note that $Q_0 \ne Q, P_1$ by Proposition \ref{Basics}(3)(1) and Fact \ref{Galois covering}(3). 
Since $\gamma$ fixes three distinct points $P_1, Q, Q_0$ on the line $\overline{P_1Q}$ and $\overline{P_1Q}$ is given by $Z=0$, we have $b=1$. 

By the above discussion and the condition (h2), the order of $H(C)$ is at most $d-2=q-1$. 
We consider a group homomorphism from $H(C)$ to $\overline{P_1P_2} \cong \mathbb P^1$ given by restrictions. 
Then, this is injective by the above discussion. 
It follows from Lemma \ref{automorphisms of P^1}(2) that $H(C)$ is cyclic. 
\end{proof}

We consider a set $S=\{\tau_3\tau_i|3 \le i \le d\}$. 
Then, $S \subset H(C)$ by Proposition \ref{Basics}(5)(1). 
Since the cardinality of $S$ is $q-1$, $S=H(C)$ by Lemma \ref{cyclic}. 
Because $H(C)$ is cyclic, there exist $i$ such that $\tau_3\tau_i$ is a generator of $H(C)$. 
Therefore, $\tau_3\tau_i$ is given by a matrix
$$ A_{\tau_3\tau_i}=\left(\begin{array}{ccc}
\zeta^2 & 0 & 0 \\
0 & \zeta & 0 \\
0 & 0 & 1
\end{array}\right), $$
where $\zeta$ is a primitive $(q-1)$-th root of unity. 
We denote $\tau_3\tau_i$ by $\gamma$.

By Proposition \ref{Basics}(3), there exists an element $\sigma \in G_{P_1} \setminus \{1\}$ such that the $(1,2)$ and $(1,3)$-elements of a matrix $A_{\sigma}$ representing $\sigma$ are not zero (see also Section 2). 
For a suitable system of coordinates, we may assume that $\sigma$ is represented by a matrix
$$ A_{\sigma}=\left(\begin{array}{ccc}
1 & 1 & 1 \\
0 & 1 & 0 \\
0 & 0 & 1
\end{array}\right). $$
An automorphism $\gamma^j\sigma\gamma^{-j}$ is represented by a matrix
$$ \left(\begin{array}{ccc}
1 & \zeta^j & \zeta^{2j} \\
0 & 1 & 0 \\
0 & 0 & 1
\end{array}\right).  $$
In particular, $\gamma^j\sigma\gamma^{-j} \in G_{P_1}$ for any $j$ with $1 \le j \le q-1$. 
Since the cardinality of a set $\{\gamma^j\sigma\gamma^{-j}|1\le j \le q-1\} \subset G_{P_1}$ is $q-1$, $G_{P_1}=\{\gamma^j\sigma\gamma^{-j}|1 \le j \le q-1\} \cup \{1\}$. 
Then, a rational function $g(x,y):=\prod_{\alpha \in \mathbb F_q} (x+\alpha y+\alpha^2) \in K(x,y)$ is fixed by any element of $G_{P_1}$. 
Therefore, $g(x,y) \in K(y)$. 
By the conditions that the tangent line $T_{P_1}C=\overline{P_1Q}$ is given by $Z=0$ and that the tangent line $T_{P_2}C=\overline{P_2Q}$ is given by $X=0$, and considering the degree of $C$, $g(x,y)=cy^{q+1}$ in $K(x,y)$, where $c \in K \setminus 0$. 
Therefore, we have the equation $f(x,y)=g(x,y)+cy^{q+1}=0$. 

Finally in this section, we investigate conditions for the smoothness of $C$. 
Let $G(X, Y, Z):=Z^{q+1}g(X/Z, Y/Z)$ and let $F(X, Y, Z):=Z^{q+1}f(X/Z, Y/Z)$. 
Then, by direct computations, we have $F(Z, Y, X)=F(X, Y, Z)$. 
Since there exist exactly $d$ points contained in $C$ and a line defined by $Y=0$, such points are smooth. 
Therefore, singular points should lie on $Y \ne 0$. 
Let $h(x, z)=G(x, 1, z)$. 
We consider $h$ as an element of $K(z)[x]$. 
Then, a set $\{\alpha+\alpha^2z|\alpha \in \mathbb F_q\} \subset K(z)$, which consists of all roots of $h(x,z)=0$, forms an additive subgroup of $K(z)$. 
According to \cite[Proposition 1.1.5 and Theorem 1.2.1]{goss}, we have the following.

\begin{lemma} \label{smoothness1} 
The polynomial $h(x,z) \in K(z)[x]$ has only terms of degree equal to some power of $p$. 
Especially, $h_x(x,z)=z^q+z$, where $h_x$ is a partial derivative by $x$. 
\end{lemma}

Assume that $(x,z) \in C$ is a singular point, i.e. $h_x(x, z)=h_z(x,z)=0$. 
Then, $(x,z)$ is $\mathbb F_q$-rational by Lemma \ref{smoothness1}. 
We have $c \ne 1$ by the following. 

\begin{lemma} \label{smoothness2}
An equality $\{h(x,z)|x, z \in \mathbb F_q\}=\{0, 1\}$ holds. 
\end{lemma}

\begin{proof} 
If $z=0$, then $h(x,z)=0$.  
We fix $z_0 \in \mathbb F_q\setminus 0$.
We consider $h(x,z_0)=z_0\prod_{\alpha \in \mathbb F_q}(x+\alpha+\alpha^2z_0) \in \mathbb F_q[x]$. 
For each $\alpha \in \mathbb F_q$, there exists a unique $\beta \in \mathbb F_q$ such that $x+\alpha+\alpha^2z_0=x+\beta+\beta^2z_0$.
Therefore, the cardinality of a set $S_0:=\{\alpha+\alpha^2z_0|\alpha \in \mathbb F_q\}$ is $q/2=2^{e-1}$. 
By direct computations, we find that any element of $S_0$ is a root of a separable polynomial $h_0(x)=\sum_{i=0}^{e-1}z_0^{2^i}x^{2^i}$, which is of degree $q/2$. 
Then, $h(x,z_0)=h_0(x)^2$ as elements of $\mathbb F_q[x]$. 
Then, by direct computations, we have $h_0(x)(h_0(x)+1)=z_0(x^q+x)$ as elements of $\mathbb F_q[x]$. 
Assume $x \in \mathbb F_q$. 
Then, $h_0(x)(h_0(x)+1)=0$. 
Therefore, $h(x, z_0)=0$ or $1$. 
If we take $x \in \mathbb F_q \setminus S_0$, then $h_0(x) \ne 0$ and hence, $h(x, z_0)=1$. 
\end{proof}

\section{If-part of the proof of Theorem \ref{innerGP}} 
We use the same notation as in the previous section, $g, f, F$, and so on. Let $C$ be a plane curve given by Equation (\ref{d-Galois}) with $c \in K \setminus \{0, 1\}$. 
As in the previous section, $C$ is smooth. 
We prove $\delta(C)=d$. 
We consider the projection $\pi_{P_1}$ from $P_1=(1:0:0)$. 
Then, we have the field extension $K(x,y)/K(y)$ with $f(x, y)=g(x,y)+cy^{q+1}=0$. 
Since $(x+\alpha y+\alpha^2)+\beta y+\beta^2=x+(\alpha+\beta)y+(\alpha+\beta)^2$, we have $f(x+\alpha y+\alpha^2, y)=f(x,y)$ for any $\alpha \in \mathbb F_q$. 
Therefore, $P_1$ is Galois. 
By the symmetric property of $F(X, Y, Z)$ for $X, Z$, we find that a point $(0:0:1)$ is also inner Galois for $C$. 
We also find that there exist a tangent line $T$ such that $I_Q(C, T)=2$ for some $Q \in C \cap T$. 
Therefore, $C$ is not projectively equivalent to a Fermat curve of degree $q+1$ (see, for example, \cite{homma3}). 
According to \cite[Part III, Lemma 1 and Proposition 1]{fukasawa2}, we have $\delta(C)=d$.

\begin{remark}
A projective equivalence class of a plane curve given by Equation $(\ref{d-Galois})$ is uniquely determined by a constant $c \in K \setminus 0$.   
Therefore, infinitely many classes exist. 
Precisely, we have Lemma $\ref{class}$ below. 
\end{remark}

\begin{lemma} \label{class}
Let $a, b \in K \setminus 0$ and let $C_a$ (resp. $C_b$) be a plane curve given by Equation $(\ref{d-Galois})$ with $c=a$ (resp. $c=b$). 
If there exists a projective transformation $\phi$ such that $\phi(C_a)=C_b$, then $a=b$. 
\end{lemma} 

\begin{proof} 
Let $P_1, \ldots, P_d$ be inner Galois points for $C_a$, which are contained in a line defined by $Y=0$. 
Then, $P_1, \ldots P_d$ are also inner Galois for $C_b$. 
Since the tangent lines $T_PC_a$ and $T_PC_b$ at $P=(\alpha^2:0:1)$ with $\alpha \in \mathbb F_q$ are given by the same equation $X+\alpha Y+\alpha^2Z=0$, $T_{P_i}C_a=T_{P_i}C_b$ for $i=1, \ldots, d$. 
We may assume that $P_1=(1:0:0)$, $P_2=(0:0:1)$ and $P_3=(1:0:1)$. 
Let $Q_2=(0:1:0)$ and let $Q_3=(1:1:0)$. 
Then, $T_{P_1}C_a \cap T_{P_2}C_a=T_{P_1}C_b \cap T_{P_2}C_b=\{Q_2\}$ and $T_{P_1}C_a \cap T_{P_3}C_a=T_{P_1}C_b \cap T_{P_3}C_b=\{Q_3\}$. 
Let $\phi$ be a linear transformation such that $\phi(C_a)=C_b$. 

If $\phi(P_1)=P_i$ for some $i \ne 1$, then we take $\sigma \in G_{P_j}(C_b)$ for some $j$ such that $\sigma(P_i)=P_1$, by Fact \ref{Galois covering}(2). 
Then, $\sigma \circ \phi(P_1)=P_1$. 
Therefore, there exists a linear transformation $\phi$ such that $\phi(C_a)=C_b$ and $\phi(P_1)=P_1$. 
If $\phi(P_2)=P_i$ for some $3 \le i \le d$, then we take $\tau \in G_{P_1}(C_b)$ such that $\tau(P_3)=P_2$, by Fact \ref{Galois covering}(2). 
Then, $\tau \circ \phi(P_2)=P_2$. 
Therefore, there exists a linear transformation $\phi$ such that $\phi(C_a)=C_b$, $\phi(P_1)=P_1$ and $\phi(P_2)=P_2$. 
If $\phi(P_3)=P_i$ for some $4 \le i \le d$, then we take $\gamma \in H(C_b)$ such that $\gamma(P_i)=P_3$, where $H(C_b)$ is a group for $C_b$ discussed in the previous section. 
Then, $\gamma\circ\phi(P_3)=P_3$. 
Therefore, there exists a linear transformation $\phi$ such that $\phi(C_a)=C_b$, $\phi(P_i)=P_i$ for $i=1, 2, 3$ and $\phi(Q_j)=Q_j$ for $j=2, 3$. 
We find that $\phi$ is identity by direct computations.  
Therefore, by considering the defining equations of $C_a$ and $C_b$, we should have $a=b$. 
\end{proof}

\begin{remark} \label{c=1}
If $c=1$, then a plane curve $C$ defined by Equation $(1c)$ is parameterized as 
$ \mathbb P^1 \rightarrow \mathbb P^2: (s:1) \mapsto (s^{q+1}:s^q+s:1). $
The distribution of Galois points for this curve has been settled in \cite{fukasawa4}. 
\end{remark}

\section{Proof of Theorem \ref{outerGP} (The case where $l \ge 3$)} 
If we have two outer Galois points, then we note the following (see Section 2). 

\begin{lemma} \label{moveGalois} 
Let $P, P_2$ be outer Galois points for $C$. 
Then, any element $\gamma \in G_P$ can be extended to a linear transformation of $\mathbb P^2$, and hence $\gamma(P_2) \in \mathbb P^2$ is also outer Galois for $C$. 
\end{lemma}

Let $d=p^el$, where $e \ge 1$, $l \ge 3$ and $l$ divides $p^e-1$, and let $P=(1:0:0)$ be an outer Galois point.  
It follows from a generalization of Pardini's theorem by Hefez \cite[(5.10) and (5.16)]{hefez} and Homma \cite{homma2} that the generic order of contact for $C$ is equal to $2$, i.e. $I_R(C, T_RC) =2$ for a general point $R \in C$ (see also \cite{hefez-kleiman, homma1}).   

Let $M \subset \mathbb P^2$ be a projective line with $P \in M$. 
We consider a homomorphism $r_P[M]: G_P \rightarrow {\rm Aut}(M)$, which is induced from the restriction.
Then, the cardinality of the kernel ${\rm Ker} \ r_P[M]$ is a power of $p$.  
We denote it by $p^{v[M]}$. 
Since the kernel ${\rm Ker} \ r_P[M]$ is a subspace of $G_P$ as $\mathbb F_p$-vector spaces, we have the following diagram. 
$$\begin{array}{cccccccccc}
  &  & (\mathbb Z/p\mathbb Z)^{\oplus v[M]} & \cong & {\rm Ker} \ r_P[M] \\
  &  & \downarrow & & \downarrow \\
0 & \rightarrow & (\mathbb Z/p\mathbb Z)^{\oplus e} & \rightarrow & G_P & \rightarrow & \langle \zeta \rangle & \rightarrow & 1 \\
  &  & \downarrow & & \downarrow & & \parallel \\
0 & \rightarrow & (\mathbb Z/p\mathbb Z)^{\oplus e-v[M]} & \rightarrow & {\rm Im} \ r_P[M] & \rightarrow & \langle \zeta \rangle & \rightarrow & 1 
\end{array} $$
Using lower splitting exact sequence as groups, we have the following. 

\begin{lemma} \label{fixed order}
The integer $l$ divides $p^{e-v[M]}-1$ for any line $M$ with $P \in M$. 
\end{lemma}

Hereafter in this section, we assume that $P_2 \in \mathbb P^2 \setminus \{P\}$ is an outer Galois point for $C$. 

\begin{proposition} \label{fixed locus}
Assume that $l \ge 3$. 
Then: 
\begin{itemize} 
\item[(1)] $v[\overline{PP_2}]=e$, and there exists a unique point $Q \in \mathbb P^2$ with $Q \ne P$ such that $\gamma(Q)=Q$ for any $\gamma \in G_P$. 
\end{itemize}
Let $Q$ be the point as in $(1)$. Furthermore, we have the following. 
\begin{itemize}
\item[(2)] If $l \ge 5$, then $P_2=Q$.  
\item[(3)] If $l=4$ and $P_2 \ne Q$, then $Q \in C$ or there exist two outer Galois point $P_3, P_4$ such that $\gamma(P_4)=P_4$ for any $\gamma \in G_{P_3}$. 
\item[(4)] If $l=3$ and $P_2 \ne Q$, then $Q \in C$. 
\end{itemize}
\end{proposition}

\begin{proof} 
Let $\gamma \in G_P \setminus \mathcal{K}_P$ and let $L_{\gamma}$ be the line, which is a fixed locus, defined as in Section 2. 
The set $C \cap L_{\gamma}$ consists of $d$ points, because $T_RC=\overline{PR}\ne L_{\gamma}$ if $R \in C \cap L_{\gamma}$ by Fact \ref{Galois covering}(3) and Lemma \ref{index}(2). 
Let $\tau \in \mathcal{Q}_P$ be a generator and let $L_{\tau}$ be the line, defined as in Section 2.  
We denote $v[\overline{PP_2}]$ by $v$ and assume that $v<e$. 

We consider the case where $\gamma(P_2)=P_2$ for some $\gamma \in G_P\setminus \mathcal{K}_P$. 
Let $\sigma \in \mathcal{K}_{P_2}$. 
Then, $\sigma(R) \in L_{\gamma}$ and $\sigma(R) \ne R$ if $R \in C \cap L_{\gamma}$, by Fact \ref{Galois covering}(1)(3) and that $L_{\gamma}$ consists of exactly $d$ points. 
Furthermore, $\sigma(P) \in T_{\sigma(R_1)}C \cap T_{\sigma(R_2)}C=\{P\}$, if $R_1, R_2 \in C \cap L_{\gamma}$ with $R_1 \ne R_2$.  
Therefore, we should have $\sigma(P)=P$. 
This is a contradiction to $v < e$.  

We consider the case where $\gamma(P_2) \ne P_2$ for any $\gamma \in G_P \setminus \mathcal{K}_P$. 
Assume that $\gamma_1(P_2)=\gamma_2(P_2)$ for $\gamma_1, \gamma_2 \in G_P$. 
Note that any element $\gamma \in G_P$ is represented as $\gamma=\sigma \tau^i$ for some $\sigma \in \mathcal{K}_P$ and some $i$ (see Section 2). 
Let $\gamma_1=\sigma_1\tau^i$ and $\gamma_2=\sigma_2\tau^j$, where $\sigma_1, \sigma_2 \in \mathcal{K}_P$. 
Since $(\tau^{-j}\sigma_2^{-1}\sigma_1\tau^j)\tau^{i-j}(P_2)=P_2$ and $\tau^{-j}\sigma_2^{-1}\sigma_1\tau^j \in \mathcal{K}_P$, we have $i=j$ and $\gamma_2^{-1} \gamma_1 \in \mathcal{K}_P$.  
Furthermore, we have $\gamma_2^{-1}\gamma_1 \in {\rm Ker} \ r_P[\overline{RP}]$, since $\gamma_2^{-1}\gamma_1(P)=P$ and $\gamma_2^{-1}\gamma_1(P_2)=P_2$. 
Therefore, we have $p^{e-v}l+1$ outer Galois points on the line $\overline{PP_2}$, by Lemma \ref{moveGalois} and that the group ${\rm Im} \ r_P[\overline{PP_2}]$ is isomorphic to $(\mathbb Z/p\mathbb Z)^{\oplus e-v} \rtimes \langle \zeta \rangle$.   

Let $R \in C \cap L_{\tau}$.  
We consider points on the line $\overline{PR}$. 
Let $\sharp G_P(R)=p^bl$, where $G_P(R)$ is the stabilizer subgroup at $R$. 
Then we have $p^{e-b}$ flexes of order $(\sharp G_P(R)-2)$ by Fact \ref{Galois covering}(3). 
We note that $(p^{b}l-2)p^{e-b} \ge p^e(l-2)$. 
Furthermore, for each outer Galois points, we spent at least degree $(d-1)(p^e(l-2))$ as the degree of the Wronskian divisor. 
Therefore, it follows from the degree of Wronskian divisor (\cite[Theorem 1.5]{stohr-voloch}) that
$$ (p^{e-v}l+1)(d-1)p^e(l-2) \le 3d(d-2). $$
Then, we have
$$ (p^{e-v}l+1)p^e(l-2) < 3d =3p^el.$$
Therefore, $(p^{e-v}l+1)(l-2)-3l<0 $. 
Note that $p^{e-v}-1 \ge l$ by Lemma \ref{fixed order}. 
Then, $(l^2+l+1)(l-2)-3l <0$. 
This is a contradiction. 
Therefore, $v=e$. 

In particular, the group ${\rm Im} \ r_P[\overline{PP_2}]$ is a cyclic group of order $l$.  
By Lemma \ref{automorphisms of P^1}(3) in Section 2, a fixed point by the group ${\rm Im} \ r_P[\overline{PP_2}]$ which is different from $P$  is uniquely determined.
We denote it by $Q$.    
Then, $\gamma(Q)=Q$ for any $\gamma \in G_P$, since $\gamma=\sigma\tau^i$ for some $\sigma \in \mathcal{K}_P$ and some $i$.  
We have (1).

We prove (2). 
Assume that $P_2 \ne Q$. 
Since the group ${\rm Im} \ r_P[\overline{PP_2}]$ is a cyclic group of order $l$, we have $l+1$ outer Galois points on the line $\overline{PP_2}$, by Lemma \ref{moveGalois}. 
Furthermore, for each outer Galois point, we spent at least degree $(d-1)p^e(l-2)$ as the degree of the Wronskian divisor, similarly to the proof of (1). 
Therefore, it follows from the degree of Wronskian divisor (\cite[Theorem 1.5]{stohr-voloch}) that
$$ (l+1)(d-1)p^e(l-2) \le 3d(d-2). $$
Then, we have
$$ (l+1)p^e(l-2) <3d=3p^el.$$
Therefore, $(l+1)(l-2)-3l<0 $. 
Then, $l^2-4l-2 <0$. 
We have $l \le 4$.

We prove (3). 
Assume that $P_2 \ne Q$ and $Q \not\in C$. 
Since ${\rm Im} \ r_P[\overline{PP_2}]$ is a cyclic group of order $l$, the cardinality of $C \cap \overline{PP_2}$ is equal to $l$ and there exists $l+1$ outer Galois points on $\overline{PP_2}$, by Fact \ref{Galois covering}(3), Lemma \ref{moveGalois} and the assumption.  
Let $C \cap \overline{PP_2}=\{R_1, \ldots, R_l\}$ and let $P, P_2, \ldots, P_{l+1}$ be outer Galois points.

Let $l=4$.  
The restriction $r_P[\overline{PP_2}](\tau)$ of the generator $\tau \in \mathcal{Q}_P$ is a generator of ${\rm Im} \ r_P[\overline{PP_2}]$.
We may assume that $\tau(R_i)=R_{i+1}$ for $i=1, 2, 3, 4$, where $R_{5}=R_1$.   
We can take $\eta_j \in {\rm Im} \ r_{P_j}[\overline{PP_2}]$ such that $\eta_j(R_1)=R_2$ for $j=2, 3, 4, 5$ by Fact \ref{Galois covering}(2). 
We consider the case where at least three elements of $\{\eta_j\}$ are of order $4$. 
We may assume that $\eta_2$, $\eta_3$, $\eta_4$ are of order $4$. 
Assume that $\eta_j(R_2)=R_4$ for any $j$ with $2 \le j \le 4$. 
Then, we have $\eta_j(R_4)=R_3$. 
Since three points on the line $\overline{PP_2}$ has the same images under $\eta_2, \eta_3, \eta_4$, these are the same automorphism of the line $\overline{PP_2}$ by Lemma \ref{automorphisms of P^1}(1). 
Then, $\eta_j$ fixes $P_2, P_3, P_4$ for $j=2, 3, 4$, because $\eta_j(P_j)=P_j$.  
This implies that $\eta_j$ is identity on $\overline{PP_2}$, by Lemma \ref{automorphisms of P^1}(1). 
This is a contradiction. 
Therefore, there exists $j$ such that $\eta_j(R_2)=R_3$. 
Then, we have $\eta_j(R_3)=R_4$. 
Therefore, $\tau$ coincides with $\eta_j$ on the line $\overline{PP_2}$. 
Then, $\tau(P_j)=\eta_j(P_j)=P_j \ne Q$.  
This implies that $\tau$ fixes $P_1, P_j$ and $Q$. 
This is a contradiction. 

We consider the case where there exist distinct $j, k$ such that $\eta_j$ and $\eta_k$ is of order $2$. 
Then, $\eta_j(R_2)=R_1$, $\eta_j(R_3)=R_4$ and $\eta_j(R_4)=R_3$. 
This holds also for $\eta_k$. 
Then $\eta_j=\eta_k$ on the line $\overline{PP_2}$ by Lemma \ref{automorphisms of P^1}(1). 
Then, $\eta_j(P_k)=\eta_k(P_k)=P_k$. 
Since the group ${\rm Im} \ r_{P_j}[\overline{PP_2}]$ is cyclic, $\eta(P_k)=P_k$ for any $\eta \in {\rm Im} \ r_{P_j}[\overline{PP_2}]$, by Lemma \ref{automorphisms of P^1}(3). 
If we take $j=3$ and $k=4$, then we have the conclusion, since any element of $G_{P_j}$ is a product of elements of $\mathcal{K}_{P_j}$ and of $\mathcal{Q}_{P_j}$. 

We prove (4). 
Let $l=3$. 
Assume that $P_2 \ne Q$ and $Q \not\in C$. 
We may assume that $\tau \in G_P$ satisfies that $\tau(R_i)=R_{i+1}$ for $i=1,2, 3$, where $R_{4}=R_1$.   
We can take $\eta \in G_{P_2}$ such that $\eta(R_1)=R_2$, by Fact \ref{Galois covering}(2).
Then, we have $\eta(R_2)=R_3$ and $\eta(R_3)=R_1$.  
This implies that $\tau$ coincides with $\eta$ on $\overline{PP_2}$, by Lemma \ref{automorphisms of P^1}(1). 
Therefore, $\tau(P_2)=\eta(P_2)=P_2 \ne Q$. 
This is a contradiction. 
\end{proof}

Let $Q \in \mathbb P^2 \setminus \{P\}$ be the point such that $\gamma(Q)=Q$ for any $\gamma \in G_P$, as in Proposition \ref{fixed locus}.
We may assume that $Q=(0:1:0)$ for a suitable system of coordinates. 
Then, the line $\overline{PQ}=\overline{PP_2}$ is defined by $Z=0$. 
Using Proposition \ref{fixed locus}(1), we can determine the defining equation of $C$, as follows. 

\begin{proposition}\label{equation1}
The curve $C$ is projectively equivalent to a plane curve whose defining equation  is of the form $ f(x,y)=(\sum_{0 \le m \le e} \alpha_mx^{p^m})^l+h(y)=0, $
where $\alpha_{e}, \ldots, \alpha_0 \in K$ and $h(y) \in K[y]$ is a polynomial. Furthermore, $\alpha_e\alpha_0 \ne 0$, the derivative $h'(y)$ is of degree $d-2$, and polynomials $h(y)$ and $h'(y)$ do not have a common root. 
\end{proposition}

\begin{proof} 
Let $\sigma \in \mathcal{K}_P$ and let $\tau \in \mathcal{Q}_P$ be a generator, as in Section 2. 
We may assume that $\tau^*(x)=\zeta x$ and $\tau^*y=y$ for $\tau^*: K(C) \rightarrow K(C)$, where $\zeta$ is a primitive $l$-th root of unity.  
Let $A_{\sigma}$ be a matrix representing $\sigma \in \mathcal{K}_P$ as in Section 2.  
Since $L_{\sigma}$ is defined by $Z=0$, the $(1,2)$-element of $A_{\sigma}$ is zero.  
Since the group $\mathcal{K}_P$ is a $\mathbb F_p(\zeta)$-vector space, we have a system of basis $b_1, \ldots, b_m$, where $km=e$. 
For any $\sigma \in \mathcal{K}_P$, the $(1,3)$-element of $A_{\sigma}$ is given by $\alpha_1b_1+\dots+\alpha_mb_m$ for some $(\alpha_1, \ldots, \alpha_m) \in \oplus^m\mathbb F_p(\zeta)$.  
We define 
$g_0(x)=\prod_{(\alpha_1, \ldots, \alpha_m)}(x+\Sigma_i \alpha_ib_i)$,  
where the subscript $(\alpha_1, \ldots, \alpha_m) \in \oplus^m\mathbb{F}_p(\zeta)$ is taken over all elements. 
Let $g=g_0^l$. 
Then, we find easily that $\gamma g(x)=g(x)$ for any element $\gamma \in G_P$. 
Therefore, there exists an element $h(y) \in K(y)$ such that $g(x)+h(y)=0$ in $K(C)$.
Then, $h(y)$ is a polynomial of degree at most $d$ by considering the degree of $C$. 
On the other hand, a set $\{\sum_i\alpha_ib_i|\alpha_i \in \mathbb F_p(\zeta)\} \subset K$, which consists of all roots of $g_0(x)=0$, forms an additive subgroup of $K$. 
According to \cite[Proposition 1.1.5 and Theorem 1.2.1]{goss},   
the polynomial $g_0$ has only terms of degree equal to some power of $p$, i.e. $g_0=\alpha_ex^{p^e}+\dots+\alpha_1x^p+\alpha_0x$ for some $\alpha_{e}, \ldots, \alpha_0 \in K$.  
Since $g_0$ is separable and has $p^e$ roots, we have $\alpha_e\alpha_0 \ne 0$. 

Finally, we prove that the degree of $h'(y)$ is $d-2$, and $h(y)$ and $h'(y)$ do not have a common root. 
Since $h(y)$ is of degree at most $d=p^el$, $h'(y)$ is of degree at most $d-2$. 
Let $F(X, Y, Z)=f(X/Z, Y/Z)Z^d$, $G_0(X, Z)=g_0(X/Z)Z^{p^e}$ and $H(Y, Z)=h(Y/Z)Z^d$. 
Then, $F_X=lG_0^{l-1}(\alpha_1Z^{p^e-1})$, $F_Y=H_Y$ and $F_Z=lG_0^{l-1}(\alpha_1XZ^{p^e-2})+H_Z$. 
We have $F_X(X, Y, 0)=0$. 
Since $d=p^el$, $F_Y(X, Y, 0)=H_Y(Y, 0)=0$. 
Assume that $h'(y)$ is of degree at most $d-3$. 
Then, $F_Z(X, Y, 0)=0+H_Z(Y, 0)=0$. 
Therefore, $C$ has singular points on the line defined by $Z=0$. 
This is a contradiction to the smoothness of $C$. 
On the other hand, if there exist $b \in K$ such that $h(b)=h'(b)=0$, then a point $(a:b:1)$ with $g_0(a)=0$ is a singular point. 
\end{proof}

\begin{lemma} \label{non-Galois} 
Let $C$ be a plane curve given by the equation as in Proposition $\ref{equation1}$.    
Then, $Q \in \mathbb P^2 \setminus C$ and $Q \ne P_2$. 
\end{lemma}

\begin{proof}
It follows from Lemma \ref{fixed locus}(1) and Fact \ref{Galois covering}(4) that $L_{\sigma}=\overline{PP_2}$ for any $\sigma \in \mathcal{K}_{P_2}$. 
Therefore, any ramified point $R \in C$ of $\pi_{P_2}$with $Z \ne 0$ is tame. 
Let $\pi_Q$ be the projection from $Q$. 
Note that $\pi_Q(x:y:1)=(x:1)$. 
By the form of $\pi_Q$, if $x-x_0$ is a local parameter at $(x_0, y_0) \in C$, then $(x_0, y_0)$ is not a ramification point. 
For a point $(x_0, y_0)$ with $f_x(x_0, y_0)=lg_0(x_0)^{l-1} \ne 0$, $y-y_0$ is a local parameter. 
Therefore, ramification points of $\pi_Q$ in $Z \ne 0$ is contained in the locus $\frac{dx}{dy}=h'(y)/f_x=0$, which is equivalent to $h'(y)=0$. 
Therefore, there exist $d-2$ lines $l_1, \ldots, l_{d-2}$ which contain $P$ and $d$ ramification points of $\pi_Q$, by Proposition \ref{equation1}. 
Since $P_2 \ne P$, for any ramification point $R$ of $\pi_Q$, the cardinality of a set $\overline{P_2R} \cap \{R'\in C|Q \in T_{R'}C\} \subset \overline{P_2R}\cap \bigcup_{i=1}^{d-2}l_i$ is at most $d-2$ .

Assume that $Q \in C$. 
It follows from Fact \ref{Galois covering}(3) that $I_Q(C, \overline{PQ})=d$. 
By Fact \ref{Galois covering}(3) again, $\gamma(Q)=Q$ for any $\gamma \in G_{P_2}$.  
Let $R \in C$ be a ramification point of $\pi_Q$ in $Z \ne 0$. 
It follows from Lemmas \ref{index}(2) and Fact \ref{Galois covering}(3) that $Q \in T_RC$. 
Since $\gamma(Q)=Q$ for any $\gamma \in G_{P_2}$, $Q \in T_{\gamma(R)}C$ for any $\gamma \in G_{P_2}$. 
Then, the cardinality of $C \cap \overline{P_2R}$ is $d$ and $Q \in T_{R'}C$ for any $R' \in C \cap \overline{P_2R}$. 
This is a contradiction to that the cardinality of $\overline{P_2R} \cap \{R'\in C|Q \in T_{R'}C\}$ is at most $d-2$. 
Therefore, $Q \in \mathbb P^2 \setminus C$. 
 
Assume that $Q \in \mathbb P^2 \setminus C$ and $Q=P_2$. 
Then, the set $C \cap \overline{PP_2}$ contains $l$ points, since the group ${\rm Im} \ r_P[\overline{PP_2}]$ is cyclic of order $l$. 
Let $\tau_2 \in \mathcal{Q}_{P_2}$ be a generator and let $L_{\tau_2}$ be the line defined as in Section 2. 
Then, the locus $\Sigma=\bigcup_{\sigma \in \mathcal{K}_{P_2}}\sigma(L_{\tau_2})$ consists of $p^e$ lines. 
By considering the order of $G_{P_2}$, the ramification locus of $\pi_Q$ in an affine plane $Z \ne 0$ is contained in the locus $\Sigma$. 
Note that the set $\bigcap_{\sigma \in \mathcal{K}_{P_2}} \sigma(L_{\tau_2})$ consists of a unique point, which is not contained in $C$, by Fact \ref{Galois covering}(3) and that the set $C \cap \overline{PP_2}$ contains two or more distinct points.  
Since the set $C \cap \sigma(L_{\tau_2})$ consists of exactly $d$ points for any $\sigma \in \mathcal{K}_{P_2}$, the number of ramification points in $Z \ne 0$ is exactly $p^e \times d$. 
On the other hand, for each $b \in K$ with $h'(b)=0$, there exist exactly $d$ points $(a, b)$ such that $f(a, b)=0$, since $\alpha_e\alpha_0 \ne 0$ and $h(b) \ne 0$ by Proposition \ref{equation1}. 
Therefore, $h'(y)$ has exactly $p^e$ roots. 
Let $R$ be a ramification point of $\pi_Q$ which is contained in $Z \ne 0$. 
Since $R$ is tame (stated above), $e_R$ is computed as the order of $\frac{dx}{dy}=h'(y)/f_x$ at $R$ plus one. 
Since $e_R=l$ for any ramification point $R \in C$ with $Z \ne 0$, the polynomial $h'(y)$ is divisible by $(y-b)^{l-1}$ if $h'(b)=0$.  
Therefore, $h'(y)$ should be of the form $c(y-b_1)^{l-1}\cdots(y-b_{p^e})^{l-1}$, which is of degree $p^e(l-1)$. 
However, $h'(y)$ is of degree $p^el-2$, by Proposition \ref{equation1}. 
This is a contradiction. 
\end{proof}

\begin{proof}[Proof of Theorem $\ref{outerGP}$ (when $l \ge 3$)]
It follows from Lemma \ref{non-Galois} that $P_2 \ne Q$ and $Q \not\in C$. 
If $l \ge 5$ or $l=3$, then this is a contradiction to Proposition \ref{fixed locus}(2)(4).  
Assume that $l=4$. 
Then, by Lemma \ref{fixed locus}(3), there exists two distinct outer Galois points $P_3, P_4$ such that $\gamma(P_4)=P_4$ for any $\gamma \in G_{P_3}$. 
Then, this is a contradiction to Lemma \ref{non-Galois}.  
\end{proof}

\section{Proof of Theorem \ref{outerGP} (The case where $l \le 2$)} 
Let $p \ge 3$, let $e \ge 1$, let $l \le 2$ and let $C$ be a smooth plane curve of degree $d=p^el \ge 4$. 
We denote by $L_{\infty} \subset \mathbb P^2$ the line defined by $Z=0$. 
Let $P \in \mathbb P^2 \setminus C$ be Galois with respect to $C$. 
Assume that $P=(1:0:0)$. 
Let $\gamma \in G_P$ and let $A_{\gamma}$ be a $3 \times 3$ matrix representing $\gamma$. 
Then, 
$$ 
A_{\gamma}=
\left(\begin{array}{ccc}
a_{11}(\gamma) & a_{12}(\gamma) & a_{13}(\gamma) \\
0 & 1 & 0 \\
0 & 0 & 1
\end{array}\right), $$
where $a_{11}(\gamma)=\pm 1$ and $a_{12}(\gamma), a_{13}(\gamma) \in K$. 
Then, $\gamma^*(x)=a_{11}(\gamma)x+a_{12}(\gamma)y+a_{13}(\gamma)$.   
Note that $\mathcal{K}_P=\{\gamma \in G_P|a_{11}(\gamma)=1\}$. 
Let $g(x,y):=\prod_{\sigma \in \mathcal{K}_P}(x+a_{12}(\sigma)y+a_{13}(\sigma))$. 
Note that the set of roots $\{a_{12}(\sigma)y+a_{13}(\sigma)| \sigma \in \mathcal{K}_P\} \subset K(y)$ forms an additive subgroup of $K(y)$. 
According to \cite[Proposition 1.1.5 and Theorem 1.2.1]{goss}, $g(x,y) \in K[y][x]$ has only terms of degree equal to some power of $p$ in variable $x$.  

Assume that $l=1$. 
Then, $\mathcal{K}_P=G_P$ and $g \in K(y)$, since $\sigma^*g=g$ for any $\sigma \in G_P$. 
There exists $h(y) \in K(y)$ such that $g(x, y)+h(y)=0$ in $K(x,y)$. 
This is a defining equation of $C$. 
Considering the degree, the rational function $h(y)$ is a polynomial.

\begin{lemma} \label{equation2}
Assume that $l=1$. 
Then, the defining equation of $C$ is of the form $g(x,y)+h(y)=0$, where $g(x,y) \in K[y][x]$ has only terms of degree equal to some power of $p$ in variable $x$.  
\end{lemma}

Assume that $\delta'(C) \ge 2$. 
Let $P_2 \in \mathbb P^2 \setminus (C\cup\{P\})$ be Galois with respect to $C$. 
By taking a suitable system of coordinates, we may assume that $P_2=(0:1:0)$. 
Then, $\overline{PP_2}=L_{\infty}$. 
Similar to the previous section, we consider a group homomorphism $r_P: G_P \rightarrow {\rm Aut}(\overline{PP_2})$, which is induced from the restriction. 
The cardinality of the kernel ${\rm Ker} \ r_P$ is a power of $p$. 
We denote it by $p^v$. 
Obviously, $0 \le v \le e$. 
Then, ${\rm Ker} \ r_P=\{\sigma \in G_P|\sigma(P_2)=P_2\}$ of $G_P$, by Lemma \ref{automorphisms of P^1}. 

\begin{lemma} \label{l=1}
If $l=1$, then $v=e$. 
\end{lemma} 

\begin{proof}
We assume that $v<e$. 
Then, the coefficient of $g(x,y) \in K[y][x]$ of degree one in variable $x$ is of degree $p^e-p^v$ in variable $y$. 
By Lemma \ref{equation2}, $p^e-p^v=p^v(p^{e-v}-1)$ is a power of $p$. 
Therefore, $p^{e-v}-1=p^b$ for some integer $b$. 
This implies $b=0$ and $p=2$. 
This is a contradiction. 
\end{proof}

By Lemmas \ref{equation2} and \ref{l=1}, we have a defining equation $g(x)+h(y)=0$, where $g, h$ has only terms of degree equal to some power of $p$. 
It is not difficult to check that $C$ is singular. 
This is a contradiction. 

Assume that $l=2$. 
Let $\tau \in G_P \setminus \mathcal{K}_P$. 
Then, $\tau(x,y)=(-x+a_{12}(\tau)y+a_{13}(\tau), y)$ for some $a_{12}(\tau), a_{13}(\tau) \in K$. 
Then, $G_P=\{\sigma \tau^i|\sigma \in \mathcal{K}_P, i=0,1\}$. 
Note that $\sigma \tau (x, y)=(-x+(a_{12}(\sigma)+a_{12}(\tau))y+(a_{13}(\sigma)+a_{13}(\tau)), y)$. 
Therefore, $\hat{g}(x, y):=\prod_{\gamma \in G_P}\gamma^*(x)=g(x, y) \times g(-x+a_{12}(\tau)y+a_{13}(\tau),y)=-g^2(x, y)-g(x,y)g(-a_{12}(\tau)y-a_{13}(\tau), y)$, since $g(x,y)$ is additive in variable $x$. 
Since $\gamma^*\hat{g}(x,y)=\hat{g}(x,y)$ for any $\gamma \in G_P$, there exists $h(y) \in K(y)$ such that $f(x,y):=\hat{g}(x, y)+h(y)=0$ in $K(x,y)$. 
This is a defining equation of $C$, and $h(y)$ is a polynomial. 

\begin{lemma} \label{equation3}
Assume that $l=2$. 
Then, the defining equation of $C$ is of the form $g^2(x,y)+g(x,y)g(ay+b, y)+h(y)+\lambda=0$, where $a, b, \lambda \in K$, $h(0)=0$ and $g \in K[y][x]$ has only terms of degree equal to some power of $p$ in variable $x$. 
\end{lemma}

We consider $P_2=(0:1:0)$. 
It follows from Lemma \ref{equation3} that there exist polynomials $g_1(x,y) \in K[x][y]$ and $h_1(x) \in K[x]$ such that $g_1(x,y)$ has only terms of degree equal to some power of $p$ in variable $y$ and $f_1(x,y):=g_1^2(x,y)+g_1(x,y)g_1(x, cx+d)+h_1(x)+\lambda_2$ is a defining polynomial of $C$ for some $c, d, \lambda_2 \in K$. 
Then, $f(x,y)=f_1(x,y)$ up to a constant. 

\begin{lemma} \label{l=2}
If $l=2$, then $v=e$. 
\end{lemma}

\begin{proof} 
Assume that $v<e$. 
Then, $g(x)$ has the term $xy^{p^e-p^v}$. 
We consider the term $x^2y^{2(p^e-p^v)}$ for $f(x,y)$, which is not zero.
Note that the polynomial $g_1(x,y)g_1(x, cx+d)$ has only terms of degree equal to some power of $p$ in variable $y$, and that 
the polynomial $g_1^2(x,y)$ has only terms of degree $p^i+p^j=p^i(1+p^{j-i})$ with $i \le j$ and $0 \le i, j\le e$ in variable $y$.
Therefore, $2p^v(p^{e-v}-1)=p^i(1+p^{j-i})$ for some $i, j$ with $i \le j$. 
Then, we should have $i=v$ and $2(p^{e-v}-1)=1+p^{j-i}$. 
This implies that $2p^{e-v}-p^{j-i}=3$. 
If $j=i$, then $p=2$. 
If $j \ne i$, then $p^{j-i}(2p^{e-v-j+i}-1)=3$. 
We should have $p=3$, $j-i=1$ and $p^{e-v-1}=1$. 
We have $i=v=e-1$ and $j=e$. 
The polynomial $f(x, y)$ also has the term $x^{p^e+1}y^{p^e-p^v}=x^{p^e+1}y^{2p^{e-1}}$. 
Since $g_1^2(x,y)$ has the term $x^{p^e+1}y^{p^{e-1}+ p^{e-1}}$, $g_1(x,y)$ has the term $x^{k_0}y^{p^{e-1}}$ for some $k_0$ with $(p^e+1)/2 \le k_0 \le p^e-p^{e-1}=2p^{e-1}$. 
Let $k$ be the highest degree $k$ such that the term $x^ky^{p^{e-1}}$ of $g_1(x,y)$ is not zero. 
Then, $k \ge k_0$ and $g_1^2(x,y)$ has the term $x^{k+(p^e-p^{e-1})}y^{p^{e-1}+1}$. 
Note that $(p^e+1)/2+(p^e-p^{e-1})=p^e+(p^e-2p^{e-1}+1)/2=p^e+(p^{e-1}+1)/2$. 
Since $g^2(x, y)$ has the term $x^{k+(p^e-p^{e-1})}y^{p^{e-1}+1}$, $k+(p^e-p^{e-1}) =p^{i_1}+p^{j_1}$ for some $i_1 \le j_1$. 
Since $p^e+(p^{e-1}+1)/2 \le p^{i_1}+p^{j_1}$, we have $j_1=e$ and $i_1=e-1$. 
Therefore, $k=2p^{e-1}$. 
Assume that $(p^e+1)/2 < 2p^{e-1}$. 
Let $k_1$ be the highest degree such that the term $x^{k_1}y^{p^{e-1}}$ of $g_1(x,y)$ is not zero and $(p^e+1)/2 \le k_1 < k$. 
Then, $g_1^2(x,y)$ and $g^2(x,y)$ have the term $x^{k+k_1}y^{2p^{e-1}}$. 
Hence, $k+k_1=p^{i_2}+p^{j_2}$ for some $i_2 \le j_2$. 
By an inequality $p^e+(p^{e-1}+1)/2 \le p^{i_2}+p^{j_2}$, we have $j_2=e$ and $i_2=e-1$, and $k=k_1$. 
This is a contradiction. 
Therefore, $(p^e+1)/2=2p^{e-1}$. 
Then, we have $e=1$.

Let $p=3$ and $e=1$. 
Then, $g(x, y)=x^3-(\alpha y+\beta)^2x$ and $g_1(x, y)=y^3-(\alpha_1 x+\beta_1)^2y$ for some $\alpha, \beta, \alpha_1, \beta_1 \in K$ with $\alpha, \alpha_1 \ne 0$. 
For a suitable system of coordinates, we may assume that $\alpha y+\beta=y$. 
Then, $g(ay+b, y)=(ay+b)^3-y^2(ay+b)=(ay+b)((a+1)y+b)((a-1)y+b)=(a^3-a)y^3-by^2+b^3$ and $g_1(x, cx+d)=(cx+d)((c+\alpha)x+(d+\beta))((c-\alpha)x+(d-\beta))$. 
If $a(a+1)(a-1) \ne 0$, then $f(x,y)$ has the term $xy^5$. 
Then, $g_1^2(x,y)$ and $g_1(x,y)g_1(x, cx+d)$ do not have this term. 
Therefore, $a(a+1)(a-1)=0$. 
We also have $c(c+\alpha)(c-\alpha)=0$. 

Assume that $g(ay+b, y)$ is of degree at most $1$. 
Then, we find that $g(ay+b, y)=0$ and we have $f(x,y)=x^6-2x^4y^2+x^2y^4+h(y)$, $f_1(x,y)=(y^3-(\alpha_1x+\beta_1)^2y)^2+(y^3-(\alpha_1x+\beta_1)y)g_1(x, cx+d)+h(x)$. 
If $g_1(x, cx+d)$ is of degree $2$ in variable $x$, then $f_1(x,y)$ has the term $x^2y^3$. 
This is a contradiction. 
Therefore, $g_1(x, cx+d)$ is of degree at most $1$.
Then, $g_1(x, cx+d)=0$.  
We also find that $h(y)$ does not have a term $y$. 
By direct computations, $C$ has a singular point on a line defined by $Y=0$. 

Assume that $g(ay+b, y)$ is of degree $2$. 
Then, $b \ne 0$ and $f(x,y)=(x^3-y^2x)^2+(x^3-y^2x)(-by^2+b^3)+h(y)$ and $f_1(x,y)=(y^3-(\alpha_1x+\beta_1)^2y)^2+(y^3-(\alpha_1x+\beta_1)^2y)g_1(x, cx+d)+h_1(x)$. 
We also find that $g_1(x, cx+d) \ne 0$. 
Then, $g_1(x, cx+d)$ is of degree $2$. 
If $\beta_1 \ne 0$, then $g(x,y)$ has the term $x^2y$. 
Since $f(x,y)$ does not have the term $x^2y$, this is a contradiction.
Therefore, $\beta_1=0$. 
Then, $f_1(x,y)=(y^3-\alpha_1^2x^2y)^2+(y^3-\alpha_1^2x^2y)g_1(x, cx+d)+h_1(x)$. 
Then, $f(x,y)$ has the term $x^3y^2$ but $f_1(x, y)$ does not have it. 
\end{proof} 

By Lemma \ref{l=2}, we have $p^v=p^e$. 
Then, $g(x,y) \in K[x]$ and $g_1(x,y) \in K[y]$. 
We denote $g(x,y)$ by $g(x)$ and $g_1(x,y)$ by $g_1(y)$. 
We have $f(x,y)=g^2(x)+g(x)(g(ay)+g(b))+\lambda_1g_1^2(y)+\lambda_2$ for some $\lambda_1, \lambda_2 \in K$. 
Let $G(X,Z)=Z^{2p^e}g(X/Z)$ and let $G_1(Y, Z)=Z^{2p^e}g_1(Y/Z)$. 
Then, $F(X, Y, Z)=Z^{2p^e}f(X/Z, Y/Z)$ $=G^2(X, Z)+G(X, Z)(G(aY, Z)+g(b)Z^{2p^e})+\lambda_1G_1^2(Y,Z)+\lambda_2 Z^{2p^e}$. 
Let $\alpha$ (resp. $\beta$) be the coefficient of $XZ^{2p^e-1}$ (resp. $YZ^{2p^e-1}$) for $G(X, Z)$ (resp. $G_1(Y, Z)$). 
Then, $F_X=2G(X,Z)\alpha Z^{2p^e-1}+\alpha Z^{2p^e-1}(G(aY, Z)+g(b)Z^{2p^e})$, $F_Y=a \alpha Z^{2p^e-1}G(X, Z)+2\lambda_1G_1(Y, Z)\beta Z^{2p^e-1}$ and $F_Z=-2G(X, Z)\alpha XZ^{2p^e-2}-\alpha XZ^{p^e-2}(G(aY, Z)+g(b)Z^{2p^e})+G(X, Z)(-a \beta YZ^{2p^e-2})-2\lambda_1G_1(Y, Z)\beta YZ^{2p^e-2}$. 
Therefore, $F_X(X, Y, 0)=F_Y(X, Y, 0)=F_Z(X, Y, 0)=0$ and we have singular points on the line $L_{\infty}$.

We have the assertion of Theorem \ref{outerGP}. 

\section{Appendix: cubic curves} 
If a smooth curve $C$ is cubic, then the following proposition may be known. 
However, we can prove this, because it seems that an explicit proof is not given in any other published papers.  

\begin{proposition} \label{cubic}
Let $C$ be a smooth plane curve of degree three. 
Then, we have the following. 
\begin{itemize}
\item[(i)] If $p=2$, then $\delta'(C)=0$ or $12$ $(=2^4-2^3+2^2)$.
\item[(ii)] If $p=3$, then $\delta'(C)=0$ or $1$. 
\item[(iii)] If $p \ne 2, 3$, then $\delta'(C)=0$ or $3$.   
\end{itemize}
Furthermore, $\delta'(C) \ge 2$ if and only if $C$ is projectively equivalent to a Fermat curve. 
\end{proposition}

Let $C$ be a smooth cubic. 
Firstly, we note the following. 
\begin{lemma} \label{extendability} 
Assume that $P \in \mathbb P^2$ is outer Galois. 
Then, any automorphism $\sigma \in G_P$ is a restriction of a linear transformation. 
\end{lemma}
\begin{proof}
The linear system given by the embedding $C \subset \mathbb P^2$ is complete and $\sigma(C \cap l)=C \cap l$ if $l$ is a line passing through $P$, by Fact \ref{Galois covering}(1). 
\end{proof} 

\begin{lemma} \label{criterion}
Assume that $p\ne 3$ and $P$ is outer Galois. 
Then, we have the following. 
\begin{itemize}
\item[(i)] There exist three inflection points whose tangent lines contain $P$. 
(A point $Q \in C$ is called an inflection point if the intersection multiplicity $I_Q(C, T_QC) \ge 3$.)
\item[(ii)] $C$ is projectively equivalent to a Fermat curve. 
\end{itemize} 
\end{lemma}

\begin{proof} 
Assume that $P=(1:0:0)$ is outer Galois. 
It follows from Lemma \ref{extendability} that a generator $\tau \in G_P$ is given by $\sigma^*z=\zeta z+ay+b$ and $\sigma^*y=y$, where $a, b \in K$ and $\zeta$ is a primitive cubic root of unity, if $p \ne 3$. 
A matrix $A_{\tau}$ representing $\tau$ is diagonalizable, in fact, if we take $$B=\begin{pmatrix} 1 & a & b \\
                     0   & 1-\zeta & 0 \\
                     0 & 0 & 1-\zeta 
\end{pmatrix},$$
then $B^{-1}A_{\tau}B$ is a diagonal matrix whose diagonal elements $\zeta, 1, 1$. 
If we take $\phi$ as a linear transformation $B^{-1}$, then $\phi(C)$ is given by $X^3+G(Y, Z)=0$, where $G$ is a homogeneous polynomial of degree three. 
If $aY+bZ$ is an irreducible component, then the line defined by $aY+bZ=0$ intersects a unique point $Q$ of $C$, which is defined by $X=aY+bZ=0$. 
Therefore, it follows from B\'{e}zout's theorem that $Q$ is an inflection point and $T_QC$ is defined by $aY+bZ=0$, which contains $P$. 
We have (i). 

Now we prove (ii). 
By the proof of (i), we may assume that $C$ is defined by $X^3+G(Y, Z)=0$. 
We use the following fact on the automorphism of $\mathbb P^1$: for any triples $(R_1, R_2, R_3)$ and $(S_1, S_2, S_3)$ given by $R_i, S_i \in \mathbb P^1$, there exist a linear transformation $\phi$ of $\mathbb P^1$ such that $\phi(R_i)=S_i$ for any $i$. 
By this fact, we have $G(Y, Z)=a(Y^3+Z^3)$ for some $a \in K$, for a suitable system of coordinates.  
Therefore, $C$ is projectively equivalent to a Fermat curve. 
\end{proof}

\begin{lemma} \label{p=3}
Assume that $p=3$. 
Then, $\delta'(C) \le 1$. 
\end{lemma}

\begin{proof}
The argument in the previous section for the case where $d=p^e$ is valid for this case.  
\end{proof}

\begin{proof}[Proof of Proposition \ref{cubic}]
Assume that $\delta'(C) \ge 1$. 
If $p=3$, then we have the assertion by Lemma \ref{p=3}.
Let $p \ne 3$. 
Then, by Lemma \ref{criterion}(ii), $C$ is projectively equivalent to a Fermat curve $X^3+Y^3+Z^3=0$. 
It is not difficult to check that three points $P_1:=(1:0:0)$, $(0:1:0)$ and $(0:0:1)$ are outer Galois. 
Therefore, $\delta'(C) \ge 3$. 
Let $P_2$ be an outer Galois point which is not equal to such three points. 
Then, $P_2$ is contained in a tangent line at an inflection point $Q$, by Lemma \ref{criterion}(i). 
Since the number of inflection points is at most $9$ $(=3d(d-2))$ (see \cite{homma1, homma3, stohr-voloch} for example), $Q$ is contained in a line $X=0$, $Y=0$ or $Z=0$. 
We may assume that $Q$ is contained in a line $X=0$. 
Then, $T_QC$ contains $P_1$ and $P_2$. 
By Lemma \ref{extendability} and considering the action of $G_{P_1}$, we have four Galois points $P_1, P_2, P_3, P_4$ on the line $T_QC$. 
Since there exists a point $Q'$ such that $P_i \in T_{Q'}C$ but $P_j \not\in T_{Q'}C$ if $i \ne j$, we have $G_{P_i} \cap G_{P_j}=\{1\}$ if $i \ne j$, by Lemma \ref{index} and Fact \ref{Galois covering}(3). 
Then, there exist $4\times 2+1=9$ automorphisms of $C$ which fixes $Q$, by Lemma \ref{index} and Fact \ref{Galois covering}(3). 
It follows from a property of an elliptic curve \cite[Theorem 10.1]{silverman} that $p=2$. 
Since Homma's result \cite{homma3} is valid for this case, we have $\delta'(C)=12$. 
\end{proof}

\
\begin{center} {\bf Acknowledgements} \end{center} 
The author was partially supported by Grant-in-Aid for Young Scientists (B) (22740001), MEXT and JSPS, Japan.  
The author thanks Professor Masaaki Homma for helpful discussions. 
The parameterization as in Remark \ref{c=1} appeared during discussions with Homma. 
The author also thanks Professor Hisao Yoshihara for helpful comments.

\end{document}